\documentclass[11pt,a4paper]{amsart}

\usepackage{mathtools,amssymb,amsthm}
\usepackage[l2tabu,orthodox]{nag} 
\usepackage[all, warning]{onlyamsmath} 
\usepackage{bm}
\usepackage{mathrsfs} 
\usepackage{stmaryrd} 
\usepackage{url}
\usepackage{mymacro}

\usepackage{color}

\definecolor{darkgreen}{rgb}{0.0, 0.6, 0.0}

\usepackage{tikz}
\usetikzlibrary{intersections, calc, arrows, cd}

\usepackage[colorlinks,citecolor=darkgreen,linkcolor=red]{hyperref}
\RequirePackage{cleveref}

\usepackage{geometry}
\numberwithin{equation}{section}
\numberwithin{figure}{section}
\numberwithin{table}{section}
\allowdisplaybreaks 

\usepackage{enumitem}
\setenumerate[1]{label={\upshape(\arabic*)},nosep}
\setenumerate[2]{label={\upshape(\alph*)},nosep}
\setitemize{nosep}
\setdescription{nosep}

\newlist{enumarabic}{enumerate}{1}
\setlist*[enumarabic]{label={\upshape(\arabic*)}, nosep}

\newlist{enumAlph}{enumerate}{1}
\setlist*[enumAlph]{label={\upshape(\Alph*)}, nosep}

\newlist{enumalph}{enumerate}{1}
\setlist*[enumalph]{label={\upshape(\alph*)}, nosep}

\newlist{enumRoman}{enumerate}{1}
\setlist*[enumRoman]{label={\upshape(\Roman*)}), nosep}

\newlist{enumroman}{enumerate}{1}
\setlist*[enumroman]{label={\upshape(\roman*)}, nosep}

 \crefname{section}{\S\!}{\S\S\!}
 \crefname{subsection}{\S\!}{\S\S\!}
 \crefname{subsubsection}{\S\!}{\S\S\!}
 \crefname{equation}{equation}{equations}

\theoremstyle{plain}
\newtheorem{theorem}{Theorem}[section]
 \crefname{theorem}{Theorem}{Theorems}
\newtheorem{theorem-definition}[theorem]{Theorem-Definition}
 \crefname{theorem-definition}{Theorem-Definition}{Theorem-Definition}
\newtheorem{lemma}[theorem]{Lemma}
 \crefname{lemma}{Lemma}{Lemmas}
\newtheorem{proposition}[theorem]{Proposition}
 \crefname{proposition}{Proposition}{Propositions}
\newtheorem{proposition-definition}[theorem]{Proposition-Definition}
 \crefname{proposition-definition}{Proposition-Definition}{Proposition-Definition}
\newtheorem{corollary}[theorem]{Corollary}
 \crefname{corollary}{Corollary}{Corollaries}
\newtheorem{claim}[theorem]{Claim}
 \crefname{claim}{Claim}{Claims}

 \crefname{question}{Question}{Questions}

 \crefname{problem}{Problem}{Problems}

 \crefname{conjecture}{Conjecture}{Conjectures}
\newtheorem{theoremA}{Theorem}
 \crefname{theoremA}{Theorem}{Theorems}

\theoremstyle{definition}
\newtheorem{definition}[theorem]{Definition}
 \crefname{definition}{Definition}{Definitions}

 \crefname{fact}{Fact}{Facts}

 \crefname{notation}{Notation}{Notations}

 \crefname{observation}{Observation}{Observations}
\newtheorem{remark}[theorem]{Remark}
 \crefname{remark}{Remark}{Remarks}

 \crefname{example}{Example}{Examples}

 \crefname{construction}{Construction}{Constructions}

\theoremstyle{remark}

\newenvironment{pfclaim}[1][$\Diamond$]
{\def\claimQED{{#1}}\noindent {\em Proof of the claim. }}
{\leavevmode\unskip\penalty9999 \hbox{}\nobreak\hfill
    \quad\hbox{\claimQED}{\smallskip}}

\def\Zar{\mathrm{Zar}}
\def\Itor{\textrm{$I$-}\mathrm{tor}}
\def\Iztor{\textrm{$I_0$-}\mathrm{tor}}

\def\Jtor{\textrm{$J$-}\mathrm{tor}}
\def\ator{\textrm{$a$-}\mathrm{tor}}

\begin{document}
\title[A characterization of perfectoid towers in terms of conormal cones]{A characterization of perfectoid towers in terms of conormal cones}

\author[K.~Hayashi]{Kazuki Hayashi}
\address{Department of Mathematics, Institute of Science Tokyo, 2-12-1 Ookayama, Meguro, Tokyo 152-8551}
\email{hazuki0694@gmail.com}



\begin{abstract}
We characterize perfectoid towers in terms of conormal cones rather than torsion parts.
This result is deduced from a refined study of the relationship between torsion with respect to a principal ideal and the associated conormal cone, building on the work of O.~Gabber and L.~Ramero.
\end{abstract}

\maketitle

\setcounter{tocdepth}{2}
\tableofcontents
\section{Introduction}

The theory of perfectoid spaces, introduced by P.~Scholze, has played a central role in commutative algebra in mixed characteristic. In order to apply perfectoid techniques to Noetherian rings, the notion of perfectoid towers was introduced by S.~Ishiro, K.~Nakazato, and K.~Shimomoto \cite{INS25}; see also \cite{Ha26a,Ha26c,HIS26,IS25,Ish26}. The definition of perfectoid towers is based on a characterization of perfectoid rings in terms of their $p$-torsion parts, originally due to O.~Gabber and L.~Ramero \cite[Theorem 16.3.75]{GR24}; see also \cite[Theorem 3.50]{INS25}. This characterization can in turn be deduced from the following one in terms of conormal cones.

\begin{theorem}[Gabber--Ramero]
\label{thm:GR16.3.64}
Let $A$ be a commutative ring. Then $A$ is a perfectoid ring \emph{(}in the sense of \cite[Definition 3.5]{BMS1}\emph{)} if and only if $A$ contains an element $\varpi$ satisfying the following conditions.
  \begin{enumerate}
  \item $p\in\varpi^pA$, and $A$ is $\varpi$-adically complete and separated.
  \item The graded ring homomorphism
  \[
  \gr_{(\varpi)}^\bullet(A)\to \gr_{(\varpi^p)}^\bullet(A);\quad x\ \mathrm{mod}\ (\varpi)^{n+1} \mapsto x^p\ \mathrm{mod}\ (\varpi^p)^{n+1}
  \]
  is an isomorphism. \emph{(}In particular, the absolute Frobenius of $A/\varpi^pA$ induces an isomorphism $A/\varpi A\xr{\cong} A/\varpi^pA$.\emph{)}
  \end{enumerate}
\end{theorem}

For the proof, we refer to \cite[Theorem 16.3.64]{GR24}, which treats more general topological rings; to show the ``if'' part of \cref{thm:GR16.3.64}, note that $A$ is a P-ring in the sense of \cite[Definition 16.2.1]{GR24}, when endowed with the $\varpi$-adic topology.
Although $\gr_{(\varpi)}^\bullet(A)$ is nothing but the polynomial ring $(A/\varpi A)[T]$ when $A$ is $\varpi$-torsion free (i.e., $\varpi$ is a non-zero-divisor of $A$), it still captures the structure in the presence of $\varpi$-torsion and thus provides a unified treatment of both cases. For instance, \cref{thm:GR16.3.64} can be applied to deduce the reducedness of perfectoid rings (\cite[Corollary 16.3.63 (i)]{GR24}) and their stability under weakly \'{e}tale base change (\cite[Theorem 16.7.1]{GR24}). 

In this context, we study perfectoid towers via conormal cones rather than torsion parts, building on the methods of Gabber--Ramero. See \cref{def:perfectoidtower} for the precise definition of perfectoid towers. We first prove that conormal cones are invariant under tilting (\cref{prop:perfdtowergr}):
\[
\gr_{I_0^{s.\flat}}^\bullet(R^{s.\flat}) \xr{\cong} \gr_{I_0}^\bullet(R).
\]
As for perfectoid rings, the corresponding result was shown in \cite[Claim 16.3.62 (ii)]{GR24}:
\[
\gr_{(\varpi^\flat)}^\bullet(A^\flat)\xr{\cong}\gr_{(\varpi)}^\bullet(A).
\]
We further establish the following tower-theoretic analogue of \cref{thm:GR16.3.64}:

\begin{theoremA}[{\cref{thm:gGr}}]
\label{thm:A}
Let $(R,I_0)$ be the pair consisting of a commutative ring $R$ and a principal ideal $I_0\subset R$. Let $\textrm{{\boldmath $R$}}=(R_0\to R_1\to\cdots\to R_i\to\cdots)$ be a tower of rings. Then {\boldmath $R$} is a perfectiod tower arising from a pair $(R,I_0)$ if and only if $R_1$ contains a principal ideal $I_1$ satisfying the following conditions.
  \begin{enumerate}
  \setcounter{enumi}{-1}
  \item For every $i\geq 0$, the induced map $R_i/I_0R_i\to R_{i+1}/I_0R_{i+1}$ is injective, the absolute Frobenius of $R_{i+1}/I_0R_{i+1}$ factors as $R_{i+1}/I_0R_{i+1}\xr{F_i}R_i/I_0R_i \hookrightarrow R_{i+1}/I_0R_{i+1}$, and $I_0(R_i)_{\Iztor}=(0)$.
  \item $I_1^p=I_0R_1$, and $R_i$ is $I_0$-adically Zariskian for any $i\geq 0$.
  \item For every $i\geq 0$, there exists an isomorphism of graded rings
  \[
  \Phi_i^\bullet\colon \gr_{I_1}^\bullet(R_{i+1}) \xr{\cong} \gr_{I_0}^\bullet(R_i)
  \]
  such that the diagram of rings
  \[
  \begin{tikzcd}
  R_{i+1}/I_0R_{i+1} \rar["F_i"] \dar & R_i/I_0R_i \\
  R_{i+1}/I_1R_{i+1} \ar[ru,"\Phi_i^0"',"\cong" sloped]
  \end{tikzcd}
  \]
  commutes, where the vertical arrow is the canonical projection.
  \end{enumerate}
\end{theoremA}

Using the theorem, we prove the reducedness of perfectoid towers (\cref{cor:reduced}) and their stability under weakly \'{e}tale base change (\cref{thm:bc}). Moreover, we deduce the tilting invariance of Krull dimension (\cref{cor:Krulldim}). Although these three results were already established in \cite{Ha26a}, the proof there rely on a reduction to the $I_0$-torsion free case, whereas our approach treats the general case directly by working with conormal cones rather than $I_0$-torsion parts.
We also include an appendix discussing the universal property of Zariskizations of pairs.

\addtocontents{toc}{\protect\setcounter{tocdepth}{-2}} 

\subsection*{Acknowledgements} 
The author would like to thank Shinnosuke Ishiro, Ryo Ishizuka, Kei Nakazato, Kazuma Shimomoto, Tatsuki Yamaguchi for their continuous support.

\subsection*{Notations and conventions}
  \begin{itemize}
  \item We consistently fix a prime number $p>0$.
  \item All rings are assumed to be commutative and unital (unless otherwise stated).
  \item For an $\F_p$-algebra $R$, let $\varphi=\varphi_R\colon R\to R$ denote the absolute Frobenius.
  \item For a ring $A$ and an ideal $I\subset A$, when we say an $A$-module $M$ is $I$-adically complete, we always mean that $M$ is Hausdorff complete with respect to the $I$-adic topology.
  \end{itemize}

\addtocontents{toc}{\protect\setcounter{tocdepth}{2}} 
\section{Torsion parts and conormal cones}

In this section, we discuss the relationship between torsion parts and conormal cones.

Let us first recall some basic notions and terminology.
By a \emph{pair} we simply mean a couple $(A,I)$ consisting of a ring $A$ and an ideal $I$ of $A$. When the ideal $I$ is principal, say $I=(a)$, then we often write $(A,a)$ in place of $(A,I)$. For an $A$-module $M$, we set
\[
M[I]\coloneqq 0:_MI=\{x\in M\mid Ix=0\}.
\]
We say that an element $x\in M$ is \emph{$I$-torsion} if for all $a\in I$ there exists an integer $n>0$ such that $a^nx=0$. Let $M_{\textrm{$I$-}\mathrm{tor}}$ denote the $A$-submodule of $M$ consisting of all $I$-torsion elements in $M$. We say that $M$ is \emph{$I$-torsion free} if $M_{\textrm{$I$-}\mathrm{tor}}=0$. Note that we follow the terminologies in \cite{FKI}. Let
\[
\varphi_{I,M}\colon M_{\Itor}\to M/M_{\Itor}
\]
denote the composition of the inclusion $M_{\Itor}\hookrightarrow M$ followed by the canonical projection $M\twoheadrightarrow M/M_{\Itor}$. Note that if $I$ is finitely generated, then $M_{\Itor}=\bigcup_{n>0}M[I^n]$.

For a pair $(A,I)$, we consider the associated \emph{conormal cone}:
\[
\gr_I^\bullet(A)=\bigoplus_{n\geq 0}I^n/I^{n+1}.
\]
If $M$ is an $A$-module (resp.\ $A$-algebra), then
\[
\gr_I^\bullet(M)=\bigoplus_{n\geq 0}I^nM/I^{n+1}M
\]
is a graded $\gr_I^\bullet(A)$-module (resp.\ $\gr_I^\bullet(A)$-algebra).

In this paper, we consider mainly the case where $I$ is a principal ideal.
The following two lemmas are useful for our later argument.

\begin{lemma}
\label{lem:grn}
Let $(A,I)$ be a pair such that $I$ is a principal ideal. Let $M$ be an $A$-module, and $n\geq 1$ an integer. Then the following conditions are equivalent.
  \begin{enumerate}
  \item For any generator $a\in A$ of $I$, the surjective $A/I$-linear map
  \begin{equation}
  \label{eq:mult}
  \gr_I^n(M) \to \gr_I^{n+1}(M)
  \end{equation}
  defined by the multiplication by $a\ \mathrm{mod}\ I^2M$ is an isomorphism.
  \item There exists a generator $a\in A$ of $I$ such that the map \eqref{eq:mult} is an isomorphism.
  \item $M[I^{n+1}]\subset M[I^n]+IM$.
  \end{enumerate}
\end{lemma}

Note that condition (3) is automatic if $M_{\Itor}=M[I]$, or equivalently, if $IM_{\Itor}=0$.

\begin{proof}
(1) $\Rightarrow$ (2): Trivial.

(2) $\Rightarrow$ (3): Pick $x\in M[I^{n+1}]$. Then $a(a^nx)=a^{n+1}x=0\in I^{n+2}M$. Hence $a^nx \in I^{n+1}M$ by assumption. Hence $a^nx=a^{n+1}y$ for some $y\in M$. Then $a^n(x-ay)=0$, and so $x=(x-ay)+ay \in M[I^n]+IM$.

(3) $\Rightarrow$ (1): Let $x\in M$ be such that $a^{n+1}x\in I^{n+2}M$. Then $a^{n+1}x=a^{n+2}y$ for some $y\in M$. Since $a^{n+1}(x-ay)=0$, we have $x-ay\in M[I^{n+1}]\subset M[I^n]+IM$. Hence $a^nx\in a^nIM=I^{n+1}M$, as desired.
\end{proof}

\begin{lemma}
\label{lem:pair ex}
Let $(A,a)$ be a pair with $a\in A$, and $M$ an $A$-module.
  \begin{enumerate}
  \item The canonical sequence of $A/aA$-modules
  \[
  M[a] \to M/aM \xr{a\ \mathrm{mod}\ a^2M} aM/a^2M \to 0
  \]
is exact.
  \item If $aM_{\ator}=0$, then the canonical sequence of $A/aA$-modules
  \[
  0 \to M_{\ator} \xr{\varphi_{(a),M}} M/aM \xr{a\ \mathrm{mod}\ a^2M} aM/a^2M \to 0
  \]
  is exact.
  \end{enumerate}
\end{lemma}

\begin{proof}
(1) If $x\in M$ satisfies $ax\in a^2M$, we have $ax=a^2y$ for some $y\in M$. Then $a(x-ay)=0$, that is, $x-ay\in M[a]$. Hence $x\ \mathrm{mod}\ aM$ lies in the image of $M[a]$.

(2) Since $M[a]=M_{\ator}$ by assumption, we only have to show that $\varphi_{(a),M}$ is injective, which follows from the assumption $aM_{\ator}=0$ again (\cite[Corollary 3.15]{INS25}).
\end{proof}

From these lemmas we deduce further results, which will be used in the next section.

\begin{lemma}
\label{lem:pair ext}
Let $(A,I)$ and $(B,J)$ be pairs such that $I$ and $J$ are principal and $IA_{\Itor}=(0)$ and $JB_{\Jtor}=(0)$. Suppose that there exists a ring homomorphism $\Phi\colon A/I\hookrightarrow B/J$. Consider the following conditions.
  \begin{enumerate}
  \item There exists a homomorphism of (possibly) non-unital rings $\Phi_{\mathrm{tor}}\colon A_{\Itor}\to B_{\Jtor}$ such that $\Phi\circ\varphi_{I,A}=\varphi_{J,B}\circ\Phi_{\mathrm{tor}}$.
  \item $\Phi$ extends to a morphism of graded rings $\Phi^\bullet\colon\gr_I^\bullet(A)\to\gr_J^\bullet(B)$.
  \end{enumerate}
Then we have the implication ``\emph{(1)} $\Rightarrow$ \emph{(2)}.'' If we assume, moreover, that
  \begin{itemize}
  \item[\emph{($\ast$)}] there exist generators $a$ and $b$ of $I$ and $J$, respectively, such that $\Phi^1(a\ \mathrm{mod}\ I^2)=b\ \mathrm{mod}\ J^2$,
  \end{itemize}
then it also holds the implication ``\emph{(2)} $\Rightarrow$ \emph{(1)}.''
\end{lemma}

\begin{proof}
(1) $\Rightarrow$ (2): Choose a generator $a$ (resp.\ $b$) of $I$ (resp.\ $J$). By \cref{lem:pair ex}, we have the commutative diagram of abelian groups with exact rows
\begin{equation}
\label{eq:torgr}
\begin{tikzcd}
0 \rar & A_{\Itor} \rar["\varphi_{I,A}"] \dar["\Phi_{\mathrm{tor}}"'] & A/I \rar["a\ \mathrm{mod}\ I^2"] \dar["\Phi"] &[1cm] I/I^2 \rar \dar["\Phi^1"] & 0 \\
0 \rar & B_{\Jtor} \rar["\varphi_{J,B}"'] & B/J \rar["b\ \mathrm{mod}\ J^2"'] & J/J^2 \rar & 0,
\end{tikzcd}
\end{equation}
where $\Phi^1$ is the map induced by the universality of cokernels. Then we deduce from \cref{lem:grn} the assertion.

(2)+($\ast$) $\Rightarrow$ (1): Let $a$ and $b$ as in ($\ast$). Since $\Phi$ is a homomorphism of graded rings, we get the commutative diagram \eqref{eq:torgr} of abelian groups with exact rows, where $\Phi_{\mathrm{tor}}$ is the map induced by the universality of kernels, and it is the desired one.
\end{proof}

\begin{remark}
\label{rem:torgrisom}
One can deduce the following assertions from the proof of \cref{lem:pair ext}.
  \begin{enumalph}
  \item If $\Phi$ is injective and (1) is satisfied with the additional condition that the diagram of (possibly non-unital) rings
  \[
  \begin{tikzcd}
  A_{\Itor} \rar["\varphi_{I,A}"] \dar["\Phi_{\mathrm{tor}}"'] & A/I \dar["\Phi"] \\
  B_{\Jtor} \rar["\varphi_{J,B}"'] & B/J
  \end{tikzcd}
  \]
  is cartesian, then we can take $\Phi^\bullet$ to be injective.
  \item Assume that $\Phi$ is an isomorphism and the condition ($\ast$) is satisfied. Then the following conditions are equivalent.
  \begin{enumerate}
  \item There exists an \emph{isomomorphism} of (possibly) non-unital rings $\Phi_{\mathrm{tor}}\colon A_{\Itor}\to B_{\Jtor}$ such that $\Phi\circ\varphi_{I,A}=\varphi_{J,B}\circ\Phi_{\mathrm{tor}}$.
  \item $\Phi$ extends to an \emph{isomorphism} of graded rings $\Phi^\bullet\colon\gr_I^\bullet(A)\to\gr_J^\bullet(B)$.
  \end{enumerate}
  \end{enumalph}
\end{remark}

\begin{remark}
\label{rem:Zar*}
The condition ($\ast$) in \cref{lem:pair ext} is satisfied if $B$ is $J$-adically Zariskian (\cref{def:Zariskian}). Indeed, choose a generator $a$ of $I$, and take a lift $b\in J$ of $\Phi^1(a\ \mathrm{mod}\ I^2)\in J/J^2$. Then the $B/J$-module $J/J^2$ is generated by $b\ \mathrm{mod}\ J^2$, and hence by Nakayama's lemma, the ideal $J\subset B$ is generated by $b$.
\end{remark} 

\section{Applications to perfectoid towers}

In this section, we apply some of results obtained in the previous section to perfectoid towers. 
By a \emph{tower of rings} we simply mean an inductive system of rings $\textrm{{\boldmath $R$}}=\{R_i\}_{i\geq 0}=\{R_i,t_i\}_{i\geq 0}$ indexed by $\N$, and a \emph{morphism of towers} is a morphism of inductive systems. Recall the definition of perfectoid towers.

\begin{definition}[{\cite[Definitions 3.4 and 3.21]{INS25}, \cite[Definition 2.3]{Ha26a}}]
\label{def:perfectoidtower}
Let $(R,I_0)$ be a pair, and $\textrm{{\boldmath $R$}}=\{R_i,t_i\}_{i\geq 0}$ a tower of rings. Consider the following conditions.  \begin{itemize}
  \item[\textbf{(a)}] $R_0=R$ and $p\in I_0$.
  \item[\textbf{(b)}] For every $i\geq 0$, the ring homomorphism $\ol{t_i}\colon R_i/I_0R_i \to R_{i+1}/I_0R_{i+1}$ induced by $t_i$ is injective.
  \item[\textbf{(c)}] For every $i\geq 0$, we have $\Im(\varphi_{R_{i+1}/I_0R_{i+1}}) \subset \Im(\ol{t_i})$.
  \end{itemize}
If conditions \textbf{(a)} $\sim$ \textbf{(c)} are satisifed, then for any $i\geq 0$ there exists a unique ring homomorphism $F_i\colon R_{i+1}/I_0R_{i+1}\to R_i/I_0R_i$ such that the diagram
\[
\begin{tikzcd}
R_{i+1}/I_0R_{i+1} \rar["\varphi"] \ar[rd,"F_i"'] & R_{i+1}/I_0R_{i+1} \\
 & R_i/I_0R_i \uar["\ol{t_i}"',hookrightarrow]
\end{tikzcd}
\]
commutes. We call $F_i$ the \emph{$i$-th Frobenius projection} (of {\boldmath $R$} associated to $(R,I_0)$).
  \begin{itemize}
  \item[\textbf{(d)}] For every $i\geq 0$, the $i$-th Frobenius projection $F_i\colon R_{i+1}/I_0R_{i+1}\to R_i/I_0R_i$ is surjective.
  \item[\textbf{(e)}] For every $i\geq 0$, $R_i$ is $I_0R_i$-adically Zariskian (\cref{def:Zariskian}).
  \item[\textbf{(f)}] $I_0$ is a principal ideal, and $R_1$ contains a principal ideal $I_1$ satisfying the following conditions.
    \begin{itemize}
    \item[\textbf{(f-1)}] $I_1^p=I_0R_1$.
    \item[\textbf{(f-2)}] For every $i\geq 0$, $\Ker(F_i)=I_1(R_{i+1}/I_0R_{i+1})$.
    \end{itemize}  
  \item[\textbf{(g)}] For every $i\geq 0$, $I_0(R_i)_{\textrm{$I_0$-}\mathrm{tor}}=(0)$. Moreover, there exists a bijection $(F_i)_{\mathrm{tor}}\colon (R_{i+1})_{\textrm{$I_0$-}\mathrm{tor}}\to (R_i)_{\textrm{$I_0$-}\mathrm{tor}}$ such that the diagram
  \begin{equation*}
  \begin{tikzcd}[column sep=large]
  (R_{i+1})_{\textrm{$I_0$-}\mathrm{tor}} \rar["\varphi_{I_0,R_{i+1}}"] \dar["(F_i)_{\mathrm{tor}}"'] & R_{i+1}/I_0R_{i+1} \dar["F_i"] \\
  (R_i)_{\textrm{$I_0$-}\mathrm{tor}} \rar["\varphi_{I_0,R_i}"'] & R_i/I_0R_i
  \end{tikzcd}
  \end{equation*}
  commutes.
  \end{itemize}
Then we say that {\boldmath $R$} is
  \begin{enumerate}
  \item a \emph{purely inseparable tower arising from $(R,I_0)$} if it satisfies \textbf{(a)} $\sim$ \textbf{(c)}.
  \item a \emph{preperfectoid tower arising from $(R,I_0)$} if it satisfies \textbf{(a)} $\sim$ \textbf{(d)}, \textbf{(f)} and \textbf{(g)}.
  \item a \emph{perfectoid tower arising from $(R,I_0)$} if it satisfies \textbf{(a)} $\sim$ \textbf{(g)}.
  \end{enumerate}
\end{definition}

Given a preperfectoid tower $\textrm{{\boldmath $R$}}=\{R_i,t_i\}_{i\geq 0}$ arising from $(R,I_0)$, we have the associated tower $\textrm{{\boldmath $R$}}^\flat=\{R_i^{s.\flat},t_i^{s.\flat}\}_{i\geq 0}$ called the \emph{tilt} of {\boldmath $R$}. This is defined as follows:
  \begin{itemize}
  \item For any $i\geq 0$, let
  \[
  R_i^{s.\flat}\stackrel{\mathrm{def}}{=} \varprojlim(\cdots \xr{F_{i+2}} R_{i+2}/I_0R_{i+2} \xr{F_{i+1}} R_{i+1}/I_0R_{i+1} \xr{F_i} R_i/I_0R_i).
  \]
  We call $R_i^{s.\flat}$ the \emph{$i$-th small tilt} (of {\boldmath $R$} associated to $(R,I_0)$).
  \item For any $i\geq 0$, let $t_i^{s.\flat}\colon R_i^{s.\flat}\to R_{i+1}^{s.\flat}$ be the unique ring homomorphism such that the diagram of rings
  \[
  \begin{tikzcd}
  R_i^{s.\flat} \rar["t_i^{s.\flat}"] \dar & R_{i+1}^{s.\flat} \dar \\
  R_{i+m}/I_0R_{i+m} \rar["\ol{t_{i+m}}"] & R_{i+m+1}/I_0R_{i+m+1}
  \end{tikzcd}
  \]
  commutes for all $m\geq 0$, where the vertical arrows are the $m$-th projections.
  \end{itemize}
Moreover, let $I_0^{s.\flat}$ denote the kernel of the $0$-th projection $R_0^{s.\flat}\to R_0/I_0$.
By \cite[Lemma 3.39]{INS25}, for any $i\geq 0$ the $0$-th projection $R_i^{s.\flat}\to R_i/I_0R_i$ induces an isomorphism
\begin{equation}
\label{eq:isomres}
R_i^{s.\flat}/I_0^{s.\flat}R_i^{s.\flat}\xr{\cong} R_i/I_0R_i,
\end{equation}
which the reader should always keep in mind. Note also that $\textrm{{\boldmath $R$}}^\flat=\{R_i^{s.\flat},t_i^{s.\flat}\}_{i\geq 0}$ is a perfectoid tower arising from $(R^{s.\flat},I_0^{s.\flat})$ by \cite[Proposition 3.41]{INS25}.

The first application of \cref{lem:pair ext} is the following.

\begin{proposition}
\label{prop:grinj}
Let $\textrm{{\boldmath $R$}}=\{R_i,t_i\}_{i\geq 0}$ be a preperfectoid tower arising from a pair $(R,I_0)$. Then for any $i\geq 0$ the following hold.
  \begin{enumerate}
  \item The homomorphism of graded rings
  \[
  \gr_{I_0}^\bullet(t_i)\colon\gr_{I_0}^\bullet(R_i) \to \gr_{I_0}^\bullet(R_{i+1})
  \]
  induced by $t_i$ is injective.
  \item The injection $R_{i+1}/I_1R_{i+1} \hookrightarrow R_{i+1}/I_0R_{i+1}$ induced by the absolute Frobenius of $R_{i+1}/I_0R_{i+1}$ \emph{(}cf. condition \emph{\textbf{(f-2)}}\emph{)} extends to an injection of graded rings
  \[
  \varphi'\colon \gr_{I_1}^\bullet(R_{i+1})\hookrightarrow \gr_{I_0}^\bullet(R_{i+1}).
  \]
  \end{enumerate}
\end{proposition}

\begin{proof}
By \cite[Theorem 2.11]{Ha26a}, the diagrams
\[
\begin{tikzcd}[column sep=large]
(R_i)_{\Iztor} \rar["\varphi_{I_0,R_i}"] \dar["(t_i)_{\mathrm{tor}}"',hookrightarrow] & R_i/I_0R_i \dar["\ol{t_i}",hookrightarrow] \\
(R_{i+1})_{\Iztor} \rar["\varphi_{I_0,R_{i+1}}"] & R_{i+1}/I_0R_{i+1},
\end{tikzcd}
\qquad
\begin{tikzcd}[column sep=large]
(R_{i+1})_{\Iztor} \rar["\varphi_{I_1,R_{i+1}}"] \dar["\varphi"'] & R_{i+1}/I_1R_{i+1} \dar["\varphi'",hookrightarrow] \\
(R_{i+1})_{\Iztor} \rar["\varphi_{I_0,R_{i+1}}"] & R_{i+1}/I_0R_{i+1}
\end{tikzcd}
\]
are cartesian for any $i\geq 0$, where $(t_i)_{\mathrm{tor}}$ is the restriction of $t_i$ and $\varphi$ is given by $x\mapsto x^p$. Thus we may apply \cref{lem:pair ext} together with \cref{rem:torgrisom} (a) to the pairs ``$(R_i,I_0R_i)$ and $(R_{i+1},I_0R_{i+1})$'' and ``$(R_{i+1},I_1R_{i+1})$ and $(R_{i+1},I_0R_{i+1})$,'' respectively.
\end{proof}

As a corollary, we have the following result.

\begin{corollary}[{cf.\ \cite[Corollary 3.12]{Ha26a}}]
\label{cor:reduced}
Let $\textrm{{\boldmath $R$}}=\{R_i,t_i\}_{i\geq 0}$ be a preperfectoid tower arising from a pair $(R,I_0)$. Fix $i\geq 0$, and suppose that $R_i$ and $R_{i+1}$ are $I_0$-adically separated.
  \begin{enumerate}
  \item The transition map $t_i\colon R_i\to R_{i+1}$ is injective.
  \item The ring $R_i$ is reduced.
  \end{enumerate}
\end{corollary}

\begin{proof}
(1) follows from \cref{prop:grinj} (1) and \cite[Chapter III, \S2.8, Corollary 1]{BouAC}.

(2) By (1), it suffices to show that $R_{i+1}$ is reduced, which can be checked by using \cref{prop:grinj} (2) (cf.\ \cite[Corollary 16.3.63 (a)]{GR24}).
\end{proof}

The second application of \cref{lem:pair ext} is the following.

\begin{proposition}
\label{prop:perfdtowergr}
Let $\textrm{{\boldmath $R$}}=\{R_i,t_i\}_{i\geq 0}$ be a preperfectoid tower arising from a pair $(R,I_0)$. Then for any $i\geq 0$, the isomorphism \eqref{eq:isomres} extends to an isomorphism of graded rings
\begin{equation}
\label{eq:perfdtower2}
\gr_{I_0^{s.\flat}}^\bullet(R_i^{s.\flat})\xr{\cong}\gr_{I_0}^\bullet(R_i).
\end{equation}
\end{proposition}

\begin{proof}
Due to \cite[Theorem 3.35 (2), Lemma 3.39]{INS25}, we only have to apply \cref{lem:pair ext} together with \cref{rem:torgrisom} (b) to the pairs $(R_i^{s.\flat},I_0^{s.\flat}R_i^{s.\flat})$ and $(R_i,I_0R_i)$.
\end{proof}

Moreover, we can now prove the main theorem as follows.

\begin{theorem}
\label{thm:gGr}
Let $\textrm{{\boldmath $R$}}=\{R_i,t_i\}_{i\geq 0}$ be a purely inseparable tower arising from a pair $(R,I_0)$. Then, under \emph{\textbf{(d)}} $\sim$ \emph{\textbf{(f)}}, the condition \emph{\textbf{(g)}} is equivalent to the following one:
  \begin{itemize}
  \item[{\boldmath $(\mathbf{g}')$}] For every $i\geq 0$, $I_0(R_i)_{\Iztor}=(0)$. Moreover, there exists an isomorphism of graded rings $\Phi_i^\bullet\colon \gr_{I_1}^\bullet(R_{i+1}) \to \gr_{I_0}^\bullet(R_i)$ such that the diagram of rings
  \begin{equation}
  \label{eq:Phi0}
  \begin{tikzcd}
  R_{i+1}/I_0R_{i+1} \rar["F_i"] \dar & R_i/I_0R_i \\
  R_{i+1}/I_1R_{i+1} \ar[ru,"\Phi_i^0"',"\cong" sloped]
  \end{tikzcd}
  \end{equation}
  commutes, where the vertical arrow is the canonical projection.
  \end{itemize}
\end{theorem}

\begin{proof}
Due to condition \textbf{(e)}, we only have to apply \cref{lem:pair ext} together with \cref{rem:torgrisom} (2) to the pairs $(R_{i+1},I_1R_{i+1})$ and $(R_i,I_0R_i)$ (cf.\ \cref{rem:Zar*}).
\end{proof}

\begin{remark}
\label{rem:w/oe}
In view of the proof of \cref{thm:gGr}, we can deduce \textbf{(g)} from {\boldmath $(\mathbf{g}')$} without \textbf{(e)} if the following condition is satisfied.
  \begin{itemize}
  \item[($\ast$)] There exist generators $f_0$ and $f_1$ of $I_0$ and $I_1$, respectively, such that $F_i(f_1\ \mathrm{mod}\ I_1R_{i+1}) = f_0\ \mathrm{mod}\ I_0R_i$.
  \end{itemize}
\end{remark}

As a corollary, we have the following result.

\begin{corollary}[{cf.\ \cite[Corollary 4.14]{Ha26a}}]
\label{cor:Krulldim}
Let $\textrm{{\boldmath $R$}}=\{R_i\}_{i\geq 0}$ be a perfectoid tower arising from a pair $(R,I_0)$, and let $\textrm{{\boldmath $R$}}^\flat=\{R_i^{s.\flat}\}_{i\geq 0}$ denote its tilt. Assume that $R_i$ is a Noetherian local ring for any $i\geq 0$. Then all layers of {\boldmath $R$} and $\textrm{{\boldmath $R$}}^\flat$ have the same Krull dimension:
\[
\begin{tikzcd}[sep=small]
\dim R \rar[equal] \dar[equal] & \dim R_0 \rar[equal] \dar[equal] & \dim R_1 \rar[equal] \dar[equal] & \cdots \rar[equal] & \dim R_i \rar[equal] \dar[equal] & \cdots \\
\dim R^{s.\flat} \rar[equal] & \dim R_0^{s.\flat} \rar[equal] & \dim R_1^{s.\flat} \rar[equal] & \cdots \rar[equal] & \dim R_i^{s.\flat} \rar[equal] & \cdots
\end{tikzcd}
\]
\end{corollary}

\begin{proof}
The assertion follows from \cref{thm:gGr} and \cite[Theorem 15.7]{Mat2}.
\end{proof}

Finally, we apply \cref{thm:gGr} to show that preperfectoid towers are stable under weakly \'{e}tale base change.
For an $\F_p$-algebra $A$, let $F_*A$ denote the ring $A$ considered as an $A$-module via restriction of scalars for the absolute Frobenius.

\begin{lemma}
Let $\textrm{{\boldmath $R$}}=\{R_i,t_i\}_{i\geq 0}$ be a purely inseparable tower arising from a pair $(R,I_0)$, and fix $i\geq 0$.
  \begin{enumerate}
  \item The $i$-th Frobenius projection $F_i\colon R_{i+1}/I_0R_{i+1}\to F_*(R_i/I_0R_i)$ is $R_i/I_0R_i$-linear.
  \item Assume that {\boldmath $R$} satisfies \emph{\textbf{(f)}} and {\boldmath $(\mathbf{g}')$}. Then the graded ring isomorphism $\Phi_i^\bullet \colon \gr_{I_1}^\bullet(R_{i+1})\to F_*\gr_{I_0}^\bullet(R_i)$ in condition {\boldmath $(\mathbf{g}')$} is $R_i/I_0R_i$-linear.
  \end{enumerate}
\end{lemma}

\begin{proof}
(1) is straightforward (see \cite[Lemma 3.20 (1)]{Ha26a}).

(2) follows from (1) and the fact that $\Phi_i^\bullet$ is a graded ring homomorphism.
\end{proof}

Following \cite[Definition 3.1.1]{GR}, we say that a ring homomorphism $A\to B$ is called \emph{weakly \'{e}tale} if it is flat and the multiplication map $B\otimes_AB\to B$ is flat. Note that this property is stable under base change (\cite[Lemma 3.1.2 (i)]{GR}). For example, ind-\'{e}tale morphisms are weakly \'{e}tale. An important property of weakly \'{e}tale morphisms is the following.

\begin{theorem}[{\cite[Theorem 3.5.13]{GR}, \cite[\href{https://stacks.math.columbia.edu/tag/0F6W}{Tag 0F6W}]{stacks-project}}]
\label{thm:wetFrob}
If $A\to B$ is a weakly \'{e}tale homomorphism of $\F_p$-algebras, then the relative Frobenius $\varphi_{B/A}$ is an isomorphism.
\end{theorem}

The theorem can be extended to conormal cones in the following sense.

\begin{lemma}
\label{lem:wetGr}
Let $f\colon A\to B$ be a weakly \'{e}tale homomorphism of rings, and $I\subset A$ an ideal.
  \begin{enumerate}
  \item We get the canonical isomorphism of graded rings
  \[
  \alpha_I\colon \gr_I^\bullet(A)\otimes_AB \xr{\cong} \gr_{IB}^\bullet(B).
  \]
  \item The induced graded ring homomorphism $\gr_I^\bullet(f)\colon \gr_I^\bullet(A)\to\gr_{IB}^\bullet(B)$ is weakly \'{e}tale.
  \item If $p\in I$, then $\alpha_I$ and $\gr_I^\bullet(f)$ induce an isomorphism of graded   rings
  \[
  \beta_I\colon F_*\gr_I^\bullet(A) \otimes_AB \xr{\cong} F_*\gr_I^\bullet(A)\otimes_{\gr_I^\bullet(A)}\gr_{IB}^\bullet(B) \xr[\varphi_{\gr_{IB}^\bullet(B)/\gr_I^\bullet(A)}]{\cong} F_*\gr_{IB}^\bullet(B).
  \]
  \end{enumerate}
\end{lemma}

\begin{proof}
(1) follows from the flatness of $f\otimes_A A/I^n \colon A/I^n\to B/I^nB$ for every $n\geq 0$.

(2) Since $A\to B$ is weakly \'{e}tale, so is the base change $\gr_I^\bullet(A)\to\gr_I^\bullet(A)\otimes_AB$. But $\gr_I^\bullet(A)\otimes_AB \xr{\cong} \gr_{IB}^\bullet(B)$ because $A\to B$ is flat.

(3) follows from (1) and (2), combined with \cref{thm:wetFrob}.
\end{proof}

Now we prove the base change theorem.

\begin{theorem}[{cf.\ \cite[Theorem 3.28]{Ha26a}}]
\label{thm:bc}
Let $\textrm{{\boldmath $R$}}=\{R_i,t_i\}_{i\geq 0}$ be a preperfectiod tower arising from a pair $(R,I_0)$, and let $R\to S$ be a weakly \'{e}tale homomorphism of rings. Then $\textrm{{\boldmath $R$}}\otimes_RS=\{R_i\otimes_RS, t_i\otimes_RS\}_{i\geq 0}$ is a preperfectoid tower arising from $(S,I_0S)$.
\end{theorem}

\begin{proof}
For the verification of conditions \textbf{(a)} $\sim$ \textbf{(d)}, \textbf{(f)}, we refer to the proof of \cite[Theorem 3.28]{Ha26a}. We check the remaining condition \textbf{(g)}. By \cref{rem:w/oe}, it suffices to verify the condition {\boldmath $(\mathbf{g}')$} in \cref{thm:gGr}.
Fix $i\geq 0$. Since $R\to S$ is flat, we have $(R_i\otimes_RS)_{\textrm{$I_0S$-}\mathrm{tor}}=(R_i\otimes_RS)_{\Iztor}=(R_i)_{\Iztor}\otimes_RS$. Hence we deduce from $I_0(R_i)_{\Iztor}=(0)$ that $I_0S(R_i\otimes_RS)_{\textrm{$I_0S$-}\mathrm{tor}}=(0)$. Moreover, since $R_i\to S_i$ is flat for every $i\geq 0$, we get induced graded ring isomorphisms
\[
\alpha_{I_1R_{i+1}}\colon \gr_{I_1R_{i+1}}^\bullet(R_{i+1})\otimes_{R_{i+1}}S_{i+1} \xr{\cong} \gr_{I_1S_{i+1}}^\bullet(S_{i+1}),\quad
\beta_{I_0R_i}\colon F_*\gr_{I_0R_i}^\bullet(R_i)\otimes_{R_i}S_i \xr{\cong} F_*\gr_{I_0S_i}^\bullet(S_i)
\]
by \cref{lem:wetGr}. These maps yield a diagram of graded rings
\[
\begin{tikzcd}
\gr_{I_1R_{i+1}}^\bullet(R_{i+1})\otimes_{\ol{R_i}}\ol{S_i} \rar["\Phi_i^\bullet\otimes_{\ol{R_i}}\ol{S_i}","\cong"'] \dar["\cong"] & F_*\gr_{I_0R_i}^\bullet(R_i) \otimes_{\ol{R_i}}\ol{S_i} \dar["\cong"'] \\
\gr_{I_1R_{i+1}}^\bullet(R_{i+1})\otimes_{R_{i+1}} S_{i+1} \dar["\alpha_{I_1R_{i+1}}"',"\cong"] & F_*\gr_{I_0R_i}^\bullet(R_i) \otimes_{R_i}S_i \dar["\beta_{I_0R_i}","\cong"'] \\
\gr_{I_1S_{i+1}}^\bullet(S_{i+1}) \rar["\Psi_i^\bullet"',"\cong",dashed]  & F_*\gr_{I_0S_i}^\bullet(S_i)
\end{tikzcd}
\]
where $\Phi_i^\bullet$ is the isomorphism obtained by {\boldmath $(\mathbf{g}')$} for {\boldmath $R$}. Then we have the desired isomorphism $\Psi_i^\bullet$ obtained as the composite.
\end{proof}

We give a few examples. For a pair $(A,I)$, let $A^{\mathrm{h}}$ denote the $I$-adic henselization of $A$. If $(A,I),(B,J)$ are pairs and $A\to B$ is an integral homomorphism such that $IB=J$, then the canonical morphism $B\otimes_A A^{\mathrm{h}} \to B^{\mathrm{h}}$ is an isomorphism (\cite[\href{https://stacks.math.columbia.edu/tag/0DYE}{Tag 0DYE}]{stacks-project}). Hence we obtain the following.

\begin{corollary}
Let $\textrm{{\boldmath $R$}}=\{R_i,t_i\}_{i\geq 0}$ be a preperfectoid tower arising from a pair $(R,I_0)$ such that each $t_i\colon R_i\to R_{i+1}$ is integral. For each $i\geq 0$, let $R_i^{\mathrm{h}}$ denote the $I_0$-adic henselization of $R_i$. Then the induced tower
\[
\textrm{{\boldmath $R$}}^{\mathrm{h}} \colon\quad R^{\mathrm{h}}\to R_1^{\mathrm{h}} \to \cdots \to R_i^{\mathrm{h}}\to\cdots
\]
is a perfectoid tower arising from $(R^{\mathrm{h}},I_0R^{\mathrm{h}})$.
\end{corollary}

The similar assertion holds for Zariskizations:

\begin{corollary}
\label{cor:bc}
Let $\textrm{{\boldmath $R$}}=\{R_i,t_i\}_{i\geq 0}$ be a preperfectoid tower arising from a pair $(R,I_0)$ such that each $t_i\colon R_i\to R_{i+1}$ is integral.
For each $i\geq 0$, let $R_i^{\mathrm{Zar}}\coloneqq(1+I_0R_i)^{-1}R_i$ denote the $I_0$-adic Zariskization of $R_i$. Then the induced tower
\[
\textrm{{\boldmath $R$}}^\Zar \colon\quad R^\Zar \to R_1^\Zar \to \cdots \to R_i^\Zar \to\cdots
\]
is a perfectoid tower arising from $(R^\Zar,I_0R^\Zar)$. 
\end{corollary}

The corollary follows from \cref{thm:bc} and \cref{prop:intZar} below.

\appendix
\def\thesection{\Alph{section}}

\section{Zariskizations of pairs}

In this appendix we give some of the results on Zariskizations of pairs used in the proof of \cref{cor:bc}. We refer to \cite[Chapter 0, \S7.3(b)]{FKI}, \cite[\S7]{NS24}, and \cite[\S3]{Tan18} as basic references for generalities of Zariskizations of pairs.

Recall that by a \emph{pair} we mean a couple $(A,I)$ consisting of a ring $A$ and an ideal $I$ of $A$. A \emph{morphism of pairs} $f\colon(A,I)\to(B,J)$ is a ring homomorphism $f\colon A\to B$ that is continuous with respect to the $I$-adic topology on $A$ and the $J$-adic topology on $B$. For example, the pairs $(A,I)$ and $(A,I^n)$ for $n\geq 1$ are isomorphic to each other.
We denote by $\mathsf{Pair}$ the category of pairs.

First of all, let us recall the definition of Zariskian pairs (\cite[Chapter 0, Proposition 7.3.2]{FKI}).

\begin{definition}
\label{def:Zariskian}
Let $(A,I)$ be a pair. We say that $A$ is \emph{$I$-adically Zariskian}, or that the pair $(A,I)$ is \emph{Zariskian}, if the following equivalent conditions are satisfied.
  \begin{enumerate}
  \item $1+I\subset A^\times$.
  \item An element $a\in A$ is invertible if and only if $a\ \mathrm{mod}\ I$ is invertible in $A/I$.
  \item $I$ is contained in the Jacobson radical of $A$.
  \end{enumerate}
\end{definition}

We denote by $\mathsf{Pair}^{\mathrm{Zar}}$ the full subcategory of $\mathsf{Pair}$ consisting of Zariskian pairs.

\begin{proposition}
The inclusion functor $\mathsf{Pair}^{\mathrm{Zar}} \hookrightarrow \mathsf{Pair}$ admits the left adjoint functor
\[
(-)^{\mathrm{Zar}} \colon \mathsf{Pair} \to \mathsf{Pair}^{\mathrm{Zar}};\quad (A,I)\mapsto (A^{\mathrm{Zar}},IA^{\mathrm{Zar}})=(A^{\mathrm{Zar}}_I,IA^{\mathrm{Zar}}_I)
\]
\end{proposition}

\begin{proof}
For an arbitrary pair $(A,I)$, we set $A^{\mathrm{Zar}}=A^{\mathrm{Zar}}_I\coloneqq S^{-1}A$, where $S=1+I$ is the multiplicative subset consisting of all elements of the form $1+a$, with $a\in I$. Then it is clear that the pair $(A,I)^{\mathrm{Zar}}\coloneqq(A^{\mathrm{Zar}},IA^{\mathrm{Zar}})$ is Zariskian. 

To proceed, we need to prepare the following claim.

\begin{claim}
Let $(A,I)$ be a pair, and $n\geq 1$ an integer. Then the canonical morphism of pairs $\mathrm{can}\colon (A,I^n)^{\mathrm{Zar}}\to(A,I)^{\mathrm{Zar}}$ is an isomorphism.
\end{claim}

\begin{pfclaim}
Any element of $1+I$ of the form $1-a$, where $a\in I$, is invertible in $A^{\mathrm{Zar}}_{I^n}$ because $1-a=(1-a^n)^{-1}(1+a+a^2+\cdots+a^{n-1})$. Hence the canonical ring homomorphism $A^{\mathrm{Zar}}_{I^n}\to A^{\mathrm{Zar}}_I$ is an isomorphism, which proves the claim.
\end{pfclaim}

For a morphism of pairs $f\colon(A,I)\to(B,J)$, there exists an integer $n\geq 1$ such that $f(I^n)\subset J$. Then $f(1+I^n)\subset 1+J$, and thus $f$ induces a unique ring homomorphism $g\colon A^{\mathrm{Zar}}_{I^n}\to B^{\mathrm{Zar}}_J$. Hence, if we define $f^{\mathrm{Zar}}\colon(A,I)^{\mathrm{Zar}}\to(B,J)^{\mathrm{Zar}}$ as the composite $(A,I)^{\mathrm{Zar}}\xr{\mathrm{can}^{-1}}(A,I^n)^{\mathrm{Zar}} \xr{g} (B,J)^{\mathrm{Zar}}$, then we get a commutative diagram of pairs
\[
\begin{tikzcd}
(A,I^n) \rar["\id_A","\cong"'] \dar & (A,I) \rar["f"] \dar & (B,J) \dar \\
(A,I^n)^{\mathrm{Zar}} \rar["\mathrm{can}","\cong"'] 
& (A,I)^{\mathrm{Zar}} \rar["f^{\mathrm{Zar}}"] & (B,J)^{\mathrm{Zar}}
\end{tikzcd}
\]
This shows that $f^{\mathrm{Zar}}$ is independent of the choice of $n$, and we have a functor $(-)^{\mathrm{Zar}}$. It is clear that this gives the left adjoint functor of the inclusion functor $\mathsf{Pair}^{\mathrm{Zar}}\hookrightarrow\mathsf{Pair}$.
\end{proof}

We say that a morphism of pairs $(A,I)\to(B,J)$ is \emph{adic} if $IB$ is the ideal of definition of $B$, namely, the identity map $\id_B$ gives an isomorphism of pairs between $(B,J)$ and $(B,IB)$.

\begin{proposition}
\label{prop:intZar}
Let $f\colon(A,I)\to(B,J)$ be an adic morphism of pairs such that the ring homomorphism $f\colon A\to B$ is integral. Then the canonical morphism of pairs
\[
(B\otimes_A A^{\mathrm{Zar}}_I, I(B\otimes_A A^{\mathrm{Zar}}_I)) \to (B,J)^{\mathrm{Zar}}=(B^{\mathrm{Zar}}_{J}, JB^{\mathrm{Zar}}_{J})
\]
is an isomorphism.
\end{proposition}

\begin{proof}
We may assume that $J=IB$.
Since $f$ is integral, so is the base change $A^{\mathrm{Zar}}_I \to B\otimes_AA^{\mathrm{Zar}}_I$. Hence $B\otimes_A A^{\mathrm{Zar}}_I$ is $I(B\otimes_A A^{\mathrm{Zar}}_I)$-adically Zariskian (\cite[Lemma 3.10 (2)]{INS25}). Thus by the universal property of $(B,IB)^{\mathrm{Zar}}_{IB}$, we obtain a map $(B,IB)^{\mathrm{Zar}}_{IB} \to (B\otimes_A A^{\mathrm{Zar}}_I, I(B\otimes_A A^{\mathrm{Zar}}_I))$. This gives an inverse map by the universality of Zariskizations and pushouts.
\end{proof}


\end{document}